\documentclass[reqno, 10pt]{amsart}
\usepackage[colorlinks=true,
pdfstartview=FitV, linkcolor=blue, citecolor=green,
urlcolor=]{hyperref}
\usepackage{amsmath,amsfonts,latexsym,amssymb}
\usepackage{amssymb,amsfonts, amsmath}
\usepackage{mathrsfs}
\usepackage[latin1]{inputenc}
\usepackage[T1]{fontenc}
\usepackage{ae,aecompl}
\usepackage{braket}
\usepackage{comment}
\usepackage{amssymb,amsfonts, amsmath}
 \usepackage{mathtools} 
\usepackage[all,arc]{xy}
\usepackage{enumerate}
\usepackage{mathrsfs}
\usepackage{mathabx}
\usepackage{color}
\usepackage{graphicx}
\usepackage{subfig}
\usepackage{verbatim} 
\usepackage{amssymb,amsfonts, amsmath}
\usepackage[all,arc]{xy}
\usepackage{enumerate}
\usepackage{mathrsfs}
\usepackage{hyperref}
\usepackage{xstring}
\usepackage{amsmath}
\usepackage[makeroom]{cancel}
\usepackage{braket}
\usepackage{mathabx}
\usepackage[title]{appendix}
\newtheorem{theorem}{Theorem}[section]

\newtheorem{proposition}[theorem]{Proposition}
\newtheorem{corollary}[theorem]{Corollary}
\newtheorem{definition}[theorem]{Definition}

\newtheorem{observation}[theorem]{Observation}
\newtheorem{question}[theorem]{Question}
\renewcommand{\leq}{\leqslant}
\renewcommand{\geq}{\geqslant}
\long\def\@savemarbox#1#2{\global\setbox#1\vtop{\hsize\marginparwidth 
  \@parboxrestore\tiny\raggedright #2}}
\makeatletter
\renewcommand*\env@matrix[1][\arraystretch]{%
\edef\arraystretch{#1}%
\hskip -\arraycolsep
\let\@ifnextchar\new@ifnextchar
\array{*\c@MaxMatrixCols c}}
\makeatother
\date{\today}
\AtEndDocument{\bigskip{\footnotesize%
\textsc{Mathematical Institute, Heidelberg University, Im Neuenheimerfeld 205, 69120
Heidelberg, Germany} \par  
 \textit{E-mail address}: \texttt{pozzetti@mathi.uni-heidelberg.de }
\medskip

\textsc{CNRS and Laboratoire Alexander Grothendieck, Institut des Hautes \'Etudes Scientifiques, 
Universite Paris-Saclay, 35 route de Chartres, 91440 Bures-sur-Yvette, France } \par  
 \textit{E-mail address}: \texttt{tsouvkon@ihes.fr}}}

\date{\today}

\title{On projective Anosov subgroups of symplectic groups}
\date{\today}

\author{Maria Beatrice Pozzetti and Konstantinos Tsouvalas}

\begin{document}

\frenchspacing

\begin{abstract} We prove that a word hyperbolic group whose Gromov boundary properly contains a $2$-sphere cannot admit a projective Anosov representation into $\mathsf{Sp}_{2m}(\mathbb{C})$, $m\in \mathbb{N}$. We also prove that a word hyperbolic group which admits a projective Anosov representation into $\mathsf{Sp}_{2m}(\mathbb{R})$ is virtually a free group or virtually a surface group, a result established indepedently by Dey--Greenberg--Riestenberg. \end{abstract}

\maketitle

\section{Introduction}

Anosov representations of word hyperbolic groups into real semisimple Lie groups were introduced by Labourie in his study of the Hitchin component \cite{Labourie} and generalized by Guichard--Wienhard in \cite{GW}.
Anosov representations form a rich and stable class of discrete subgroups of real Lie groups, extensively studied by Kapovich--Leeb--Porti \cite{KLP2}, Gu\'eritaud--Guichard--Kassel--Wienhard \cite{GGKW}, Bochi--Potrie--Sambarino \cite{BPS} and others, and today are recognized as the correct higher rank generalization of convex cocompact subgroups of rank one Lie groups.

Let $\mathbb{K}=\mathbb{R}$ or $\mathbb{C}$ and an integer $1\leq k\leq \frac{d}{2}$. A representation $\rho:\Gamma \rightarrow \mathsf{SL}_d(\mathbb{K})$ of a word hyperbolic group $\Gamma$ is called \emph{$k$-Anosov} if it is Anosov with respect to the pair of opposite parabolic subgroups of $\mathsf{SL}_{d}(\mathbb{K})$ obtained as the stabilizers of a $k$-plane and a complementary $(d-k)$-plane (see \S\ref{background}). In this note, we study $1$-Anosov (or projective Anosov) representations whose images are subgroups of the symplectic group $\mathsf{Sp}_{2m}(\mathbb{K})$, $m\geq 1$. The first main result of this note is the following.

\begin{theorem}\label{symp1} Suppose that $\Gamma$ is a word hyperbolic group which admits an $1$-Anosov representation $\rho:\Gamma \rightarrow \mathsf{Sp}_{2m}(\mathbb{R})$. Then $\Gamma$ is virtually a free group or virtually a surface group. \end{theorem}

Theorem \ref{symp1} was established independently and with different techniques by Dey--Greenberg--Riestenberg in \cite{DGR}. Our second main result provides restrictions for 1-Anosov representations into complex symplectic groups.

\begin{theorem}\label{symp2} Suppose that $\Gamma$ is a word hyperbolic group which admits an $1$-Anosov representation $\rho:\Gamma \rightarrow \mathsf{Sp}_{2m}(\mathbb{C})$. Then the Gromov boundary of $\Gamma$ cannot properly contain a $2$-sphere. \end{theorem}

As a consequence of the previous two theorems and the fact that odd exterior powers of the symplectic group are symplectic we obtain the following corollary.

\begin{corollary}\label{odd}  Suppose that $\Gamma$ is a word hyperbolic group and $\rho: \Gamma \rightarrow \mathsf{Sp}_{2m}(\mathbb{K})$ is a $k$-Anosov representation where $1\leq k\leq m$ is an odd integer.

\noindent \textup{(i)} If $\mathbb{K}=\mathbb{R}$ then $\Gamma$ is virtually a surface group or a free group.\\
\noindent \textup{(ii)} If $\mathbb{K}=\mathbb{C}$ then the Gromov boundary of $\Gamma$ cannot properly contain a $2$-sphere.\end{corollary}

In particular, every word hyperbolic group which contains an infinite index quasiconvex subgroup isomorphic to a uniform lattice in $\mathsf{SL}_2(\mathbb{K})$ fails to admit a $k$-Anosov representation into $\mathsf{Sp}_{2m}(\mathbb{K})$ for $k$ odd with $1\leq k \leq m$ .

A representation $\rho:\Gamma \rightarrow \mathsf{SL}_d(\mathbb{K})$ is called {\em Borel Anosov} if $\rho$ is $k$-Anosov for every \hbox{$1\leq k\leq \frac{d}{2}$.} To our knowledge, the only known examples of word hyperbolic groups admitting Borel Anosov representation into $\mathsf{SL}_d(\mathbb{C})$, for some $d\geq 4$, are virtually isomorphic to a convex cocompact Kleinian group. The following corollary of Theorem \ref{symp2} provides several examples of word hyperbolic groups which fail to admit Borel Anosov representations into $\mathsf{SL}_d(\mathbb{C})$ for infinitely many $d\in \mathbb{N}$.

\begin{corollary}\label{2q+1} Suppose that $\Gamma$ is a word hyperbolic group whose Gromov boundary properly contains a $2$-sphere. Then there is no $(2q+1)$-Anosov representation $\rho:\Gamma \rightarrow \mathsf{SL}_{4q+2}(\mathbb{C})$.\end{corollary}

The method of proof of Theorem \ref{symp1} and Theorem \ref{symp2} allows us to obtain analogous restrictions for certain classes of hyperconvex reprsentations. The notion of a $(p,q,r)$-hyperconvex representation (see Definition \ref{hyperconvex}) was introduced by Pozzetti--Sambarino--Wienhard in \cite{PSW} and shares common tranvsersality properties with Hitchin representations. 

\begin{theorem}\label{hyp} Let $\Gamma$ be a word hyperbolic group and $1\leq p\leq \frac{d}{2}$ an odd integer. Suppose that $\rho:\Gamma \rightarrow \mathsf{GL}_{d}(\mathbb{K})$ is a $(p,p,2p)$-hyperconvex representation.\\
\noindent \textup{(i)} If $\mathbb{K}=\mathbb{R}$, then $\Gamma$ is virtually a surface group or a free group.\\
\noindent \textup{(ii)} If $\mathbb{K}=\mathbb{C}$, then the Gromov boundary of $\Gamma$ cannot properly contain a $2$-sphere.\end{theorem}

For $\mathbb{K}=\mathbb{R}$ and $d=2p$, the previous previous theorem recovers the main result of Dey \cite{Dey} in dimension $2p$ and of the second named author \cite{Tso}.

In contrast to the case of Borel Anosov representations into $\mathsf{SL}_{d}(\mathbb{R})$, where only virtually free or virtually surface groups arise in certain dimensions (see \cite{Dey}, \cite{Tso} and \cite[Thm. 1.1 \& 1.6]{CanaryT}), there is no known classification of the domain groups of Borel Anosov representations into $\mathsf{SL}_d(\mathbb{C})$ for any $d\geq 3$. In the view of Corollary \ref{2q+1}, this leads to the following analogue of Sambarino's question\footnote{Andr\'es Sambarino asked whether a torsion-free word hyperbolic group which admits a Borel Anosov representation into $\mathsf{SL}_d(\mathbb{R})$ is necessarily a free group or a surface group.} for complex special linear groups.

\begin{question}\label{complex} Suppose that $\Gamma$ is a word hyperbolic group which admits a Borel Anosov representation into $\mathsf{SL}_d(\mathbb{C})$ for some $d \geq 4$. Is $\Gamma$ virtually isomorphic to a convex cocompact subgroup of $\mathsf{SL}_2(\mathbb{C})$?\end{question}

We require here that $d\geq 4$ since the inclusion of every uniform lattice $\Delta \subset \mathsf{SU}(2,1)$ into $\mathsf{SL}_3(\mathbb{C})$ is $1$-Anosov and the Gromov boundary of $\Delta$ is homeomorphic to the $3$-sphere.

\medskip
\noindent \textbf{Acknowledgements.} The first named author acknowledges funding by the DFG, 427903332 (Emmy Noether), and is supported by the DFG under Germany's Excellence Strategy EXC-2181/1-390900948. The second named author acknowledges support from the European Research Council (ERC) under the European's Union Horizon 2020 research and innovation programme (ERC starting grant DiGGeS, grant agreement No 715982).

\section{Background}\label{background}
\hbox{In this section, we provide some notation, define Anosov representations and prove a key proposition.}

Let $\mathbb{K}=\mathbb{R}$ or $\mathbb{C}$. For $d \geqslant 2$, $(e_1,\ldots,e_d)$ is the canonical basis of $\mathbb{K}^d$ equipped with the standard Hermitian inner product $\langle \cdot,\cdot\rangle$ and denote by $\mathsf{S}^1(\mathbb{K}^d)$ the unit sphere. Throughout this note, $\mathsf{Sp}_{2m}(\mathbb{K})$ is the automorphism group of the symplectic bilinear form $\omega:\mathbb{K}^{2m}\times \mathbb{K}^{2m}\rightarrow \mathbb{K}$, $$\omega(u,v):=u^T (\Omega_mv)=\sum_{i=1}^n u_iv_{i+n}-\sum_{i=1}^{n}u_{i+n}v_i \  \ u,v\in \mathbb{K}^{2m},$$ where $\Omega_m:=\begin{pmatrix}[0.8]
0 & I_m \\ 
-I_m &0\end{pmatrix}$. 

\subsection{Anosov representations.} \label{definition} Throughout this note, $\Gamma$ denotes a word hyperbolic group and $\partial_{\infty}\Gamma$ its Gromov boundary. For more background on hyperbolic groups we refer to \cite{Gromov}. We equip $\Gamma$ with a left-invariant word metric and for $\gamma \in \Gamma$, $|\gamma|_{\Gamma}$ denotes the distance of $\gamma$ from the identity element of $\Gamma$. For a matrix $g \in \mathsf{SL}_d(\mathbb{K})$ let $\sigma_{1}(g) \geqslant \sigma_{2}(g) \geqslant \ldots \geqslant \sigma_{d}(g)$ be the singular values of $g$. Recall that for each $i$, $\sigma_{i}(g)=\sqrt{\lambda_{i}(gg^{\ast})}$, where $g^{\ast}$ is the conjugate transpose of $g$ and $\lambda_1(gg^{\ast})\geq \cdots \geq \lambda_{d}(gg^{\ast})$ are the moduli of eigenvalues of $gg^{\ast}$.

Fix $k\in \mathbb{N}$ with $1\leq k\leq \frac{d}{2}$. A representation $\rho:\Gamma \rightarrow \mathsf{SL}_d(\mathbb{K})$ of a word hyperbolic group $\Gamma$ is called {\em $k$-Anosov} if and only if there exist $c,\mu >0$ such that $$\frac{\sigma_{k}(\rho(\gamma))}{\sigma_{k+1}(\rho(\gamma))} \geqslant ce^{\mu |\gamma|_{\Gamma}} \ \ \forall \gamma \in \Gamma.$$ The fact that the above singular value gap characterization is equivalent to Labourie's original dynamical definition is due to Kapovich--Leeb--Porti in \cite{KLP2} and independently to Bochi--Potrie--Sambarino \cite{BPS}.

A $k$-Anosov representation $\rho:\Gamma \rightarrow \mathsf{SL}_d(\mathbb{K})$ admits a unique pair of $\rho$-equivariant continuous maps $\xi_{\rho}^{k}: \partial_{\infty}\Gamma \rightarrow \mathsf{Gr}_{k}(\mathbb{K}^d)$ and $\xi_{\rho}^{d-k}: \partial_{\infty}\Gamma \rightarrow \mathsf{Gr}_{d-k}(\mathbb{K}^d)$ called the {\em Anosov limit maps of $\rho$}. Two of the crucial properties of the limit maps of $\rho$ is that they are {\em compatible} (i.e. $\xi_{\rho}^k(x)\subset \xi_{\rho}^{d-k}(x)$ for every $x\in\partial_{\infty}\Gamma$) and {\em transverse}, i.e. $\mathbb{K}^d=\xi_{\rho}^k(x)\oplus \xi_{\rho}^{d-k}(y)$ when $x\neq y$. For more background on Anosov representations and their limit maps we refer the reader to \cite{Labourie, GW, KLP2, GGKW, BPS, Canary}.

We will need the following elementary observation.

\begin{observation} \label{lift} Let $X$ be a metrizable space and $\xi:X \rightarrow \mathbb{P}(\mathbb{K}^d)$ be a continuous map. Suppose that there exists $W\in \mathsf{Gr}_{d-1}(\mathbb{K}^d)$ with $\mathbb{K}^d=W\oplus f(x)$ for every $x \in X$. Then $\xi$ lifts to a continuous map $\widetilde{\xi}:X \rightarrow \mathsf{S}^1(\mathbb{K}^d)$.\end{observation}

\begin{proof} We choose $g\in \mathsf{GL}_d(\mathbb{K})$ such that $W=g\langle e_2,\ldots,e_d\rangle$. Then it is immediate to define a lift $\widetilde{\xi}$ of $\xi$: for $x\in X$ write $\xi(x)=[a_{x}ge_1+gv_{x}]$, for some $a_{x}\in \mathbb{K}\smallsetminus \{0\}$ and $v_{x}\in W$, and let $$\widetilde{\xi}(x)=\frac{ge_1+a_{x}^{-1}gv_{x}}{\big|\big|ge_1+ a_x^{-1}gv_x\big|\big|}.$$ It is straightforward to check that $\widetilde{\xi}$ is well defined, continuous and a lift of $\xi$.\end{proof}

We denote by $K_d$ the maximal compact subgroup of $\mathsf{SL}_{m}(\mathbb{K})$ preserving the inner product $\langle \cdot,\cdot\rangle$. For a point $[k_1e_1]\in \mathbb{P}(\mathbb{K}^d)$ and a hyperplane $V=k_2\langle e_1,\ldots, e_{d-1}\rangle$, where $k_1,k_2\in K_d$, their distance is defined as $\textup{dist}([k_1e_1],V):=|\langle k_1e_1,k_2e_d\rangle|$. 

We close this section with following proposition which is crucial for the proof of the theorems in the following section.

\begin{proposition}\label{lift2} Let $\Gamma$ be a non-elementary word hyperbolic group and fix $w_0\in \partial_{\infty}\Gamma$. Suppose that $\rho:\Gamma \rightarrow \mathsf{Sp}_{2m}(\mathbb{K})$ is an $1$-Anosov representation with Anosov limit map $\xi_{\rho}^1:\partial_{\infty}\Gamma \rightarrow \mathbb{P}(\mathbb{K}^{2m})$. Then there exists a continuous lift $\widetilde{\xi}_{\rho}^1:\partial_{\infty}\Gamma \smallsetminus \{w_0\}\rightarrow \mathsf{S}^1(\mathbb{K}^{2m})$ of $\xi_{\rho}^1$ \textup{(}restricted on $\partial_{\infty}\Gamma\smallsetminus \{w_0\}$\textup{)} and for every pair of distinct points $x,y\neq w_0$ we have $$\omega \big(\widetilde{\xi}_{\rho}^1(x),\widetilde{\xi}_{\rho}^1(y)\big)\neq 0.$$
\end{proposition}

\begin{proof} Let $\xi_{\rho}^{2m-1}:\partial_{\infty}\Gamma \rightarrow \mathsf{Gr}_{2m-1}(\mathbb{K}^{2m})$ be the limit map transverse to $\xi_{\rho}^1$. Note that $\xi_{\rho}^1(x)$ is transverse to the hyperplane $\xi_{\rho}^{2m-1}(w_0)$ for every $x\neq w_0$, thus, by Observation \ref{lift}, we obtain a continuous lift $\widetilde{\xi}_{\rho}^1:\partial_{\infty}\Gamma \smallsetminus \{w_0\}\rightarrow \mathsf{S}^1(\mathbb{K}^{2m})$ of $\xi_{\rho}^1$ restricted on $\partial_{\infty}\Gamma \smallsetminus \{w_0\}$.

Now if $x,y\neq w_0$ are distinct, by the compatibility of the limit maps of $\rho$, we may choose $k_x,k_y\in \mathsf{Sp}_{2m}(\mathbb{K})\cap K_{d}$ such that $$\xi_{\rho}^1(x)=[k_xe_1], \xi_{\rho}^{1}(y)=[k_ye_1] \ \textup{and}\ \xi_{\rho}^{2m-1}(y)=k_y\big \langle e_1,\ldots ,e_m,e_{m+2},\ldots,e_{2m}\big \rangle.$$ Then a straightforward computation shows that \begin{align*}\big|\omega\big(k_ye_1,k_xe_1\big)\big|&=\big|(k_ye_1)^T\Omega_m k_xe_1\big|=\big|e_1^T(k_y^T\Omega_mk_xe_1)\big|=\big|e_1^T(\Omega_mk_y^{-1}k_xe_1)\big|\\&=\big|\langle \Omega_mk_y^{-1}k_xe_1,e_1\rangle\big|=\big|\langle k_y^{-1}k_xe_1,e_{m+1}\rangle\big|=\textup{dist}\big(\xi_{\rho}^1(x),\xi_{\rho}^{2m-1}(y)\big)\end{align*} and hence by transversality we obtain $\big|\omega(\widetilde{\xi}_{\rho}^1(x),\widetilde{\xi}_{\rho}^{1}(y)\big)\big|=\big|\omega(k_xe_1,k_ye_1)\big|\neq 0$. This concludes the proof of the proposition. \end{proof}

\section{proofs} By default, we say that a representation $\rho:\Gamma \rightarrow \mathsf{Sp}_{2m}(\mathbb{K})$ is $k$-Anosov if $\rho$ is $k$-Anosov as a representation of $\Gamma$ into $\mathsf{SL}_{2m}(\mathbb{K})$.

\begin{proof}[Proof of Theorem \ref{symp1}] Suppose that there exists an $1$-Anosov representation $\rho:\Gamma \rightarrow \mathsf{Sp}_{2m}(\mathbb{R})$.

Fix $w_0\in \partial_{\infty}\Gamma$ and suppose that $\Gamma$ is not virtually a surface group or a free group. By a result of Gabai \cite{Gabai}, $\partial_{\infty}\Gamma$ cannot be the circle and by \cite{BK} we may find an embedding $\iota_1:S^1\rightarrow \partial_{\infty}\Gamma \smallsetminus \{w_0\}$. 

By Proposition \ref{lift2} there exists a lift $\widetilde{\xi}_\rho^1:\partial_{\infty}\Gamma \smallsetminus \{w_0\}\rightarrow \mathsf{S}^1(\mathbb{R}^{2m})$ with $\omega\big(\widetilde{\xi}_{\rho}^1(x),\widetilde{\xi}_{\rho}^1(y)\big)\neq 0$ when $x\neq y$. Then we consider the map $L_1:S^1\rightarrow \mathbb{R}$ defined as follows $$L_1(z):=\omega\big(\widetilde{\xi}_{\rho}^1(\iota_1(z)),\widetilde{\xi}_{\rho}^1(\iota_1(-z))\big) \ \ z\in S^1,$$ where $-z\in S^1$ denotes the antipodal point of $z$. The map $L_1$ is odd, since $\omega$ is symplectic, and everywhere non-zero. However, this contradicts the fact that $L_1$ has to have constant sign. Thus, we conclude that $\Gamma$ is virtually a surface group or virtually a free group.\end{proof}

\begin{proof}[Proof of Theorem \ref{symp2}] Suppose that there exists an $1$-Anosov representation $\rho:\Gamma \rightarrow \mathsf{Sp}_{2m}(\mathbb{C})$. We prove that $\partial_{\infty}\Gamma$ cannot properly contain $S^2$.

We argue by contradiction. Indeed, suppose that there exists $w_0\in \partial_{\infty}\Gamma$ and an embedding $\iota_2:S^2 \rightarrow \partial_{\infty}\Gamma\smallsetminus \{w_0\}$. By Proposition \ref{lift2}, the Anosov limit map $\xi_{\rho}^1$ of $\rho$ lifts to a continuous map $\widetilde{\xi}_{\rho}^1:\partial_{\infty}\Gamma \smallsetminus \{w_0\}\rightarrow \mathsf{S}^1(\mathbb{C}^{2m})$. By using this lift we define the map $L_2:S^2\rightarrow \mathbb{C}$, $$L_2(z):=\omega \big(\widetilde{\xi}_{\rho}^1(\iota_2(z)),\widetilde{\xi}_{\rho}^1(\iota_2(-z))\big) \ \ z\in S^2.$$ The map $L_2$ is continuous, $L_2(z)\neq 0$ for every $z\in S^2$ and $L_2$ is odd since $\omega$ is symplectic. However, such map $L_2$ cannot exist by the Borsuk--Ulam theorem (e.g. see \cite[Cor. 2B.7]{Hatcher}). We have reached a contradiction and hence $\partial_{\infty}\Gamma$ cannot properly contain a $2$-sphere. \end{proof}

\begin{rmk} Proposition \ref{lift2} shows that if $\rho:\Gamma \rightarrow \mathsf{Sp}_{2m}(\mathbb{K})$ is $1$-Anosov with limit map $\xi_{\rho}^1:\partial_{\infty}\Gamma \rightarrow \mathbb{P}(\mathbb{K}^{2m})$, then $|\omega (\xi_{\rho}^1(x),\xi_{\rho}^1(y))|\neq 0$ for every pair of distinct points $x,y\in \partial_{\infty}\Gamma$. In particular, if $\Gamma$ is a uniform lattice in $\mathsf{Sp}_2(\mathbb{K})\cong \mathsf{SL}_2(\mathbb{K})$, then $\xi_{\rho}^1$ cannot lift to a continuous map $\partial_{\infty}\Gamma \rightarrow \mathsf{S}^1(\mathbb{K}^{2m})$.\end{rmk}
\medskip

For $d\geq 2$ and $1\leq k\leq d$ let $\wedge^{k}: \mathsf{SL}_{d}(\mathbb{K})\rightarrow \mathsf{SL}(\wedge^{k}\mathbb{K}^{d})$ be the $k^{\textup{th}}$ exterior power. We denote by $\tau_{k}:\mathsf{Gr}_{k}(\mathbb{K}^d) \rightarrow \mathbb{P}(\wedge^k \mathbb{K}^d)$ and $\tau_{d-k}:\mathsf{Gr}_{d-k}(\mathbb{K}^d) \rightarrow \mathsf{Gr}_{d_k-1}(\wedge^{k}\mathbb{K}^d)$, $d_k=\binom{d}{k}$, the Pl${\textup{\"u}}$cker embeddings. It is a standard fact that if $\rho:\Gamma \rightarrow \mathsf{SL}_d(\mathbb{K})$ is $k$-Anosov, the pair of Anosov limit maps of the $1$-Anosov representation $\wedge^k \rho:\Gamma \rightarrow \mathsf{SL}(\wedge^k \mathbb{K}^d)$ is $(\tau_k\circ \xi_{\rho}^k, \tau_{d-k}\circ \xi_{\rho}^{d-k})$.

\begin{proof}[Proof of Corollary \ref{odd}] Observe that if $1\leq k\leq m$ is odd, the matrix $\wedge^{k}\Omega_m$ is skew symmetric and hence $\wedge^{k}\mathsf{Sp}_{2m}(\mathbb{K})$ preserves a symplectic form on $\wedge^k\mathbb{K}^{2m}$. If $\rho:\Gamma \rightarrow \mathsf{Sp}_{2m}(\mathbb{K})$ is $k$-Anosov, the exterior power $\wedge^k \rho$ is $1$-Anosov, thus, items \textup{(i)} and \textup{(ii)} follow immediately by Theorem \ref{symp1} and Theorem \ref{symp2} respectively.   \end{proof}

\begin{proof}[Proof of Corollary \ref{2q+1}] Observe that the group $\wedge^{2q+1}\mathsf{SL}_{4q+2}(\mathbb{C})$ preserves the symplectic bilinear form $\omega_0: \wedge^{2q+1}\mathbb{C}^{4q+2}\times \wedge^{2q+1}\mathbb{C}^{4q+2}\rightarrow \mathbb{C}$ given by the formula $$\omega_0(u,v)=u\wedge v \ \ u,v\in \wedge^{2q+1}\mathbb{C}^{4q+2}.$$ Now if there exists an $(2q+1)$-Anosov representation $\rho:\Gamma \rightarrow \mathsf{SL}_{4q+2}(\mathbb{C})$, its exterior power $\wedge^{2q+1}\rho$ is $1$-Anosov and preserves the symplectic form $\omega_0$. Thus, by Theorem \ref{symp2} we conclude that $\partial_{\infty}\Gamma$ cannot properly contain a $2$-sphere, contradicting our assumption.  Therefore, there is no $(2q+1)$-Anosov representation $\rho:\Gamma \rightarrow \mathsf{SL}_{4q+2}(\mathbb{C})$. \end{proof}

Before we give the proof of Theorem \ref{hyp} let us recall the definition of a $(p,q,r)$-hyperconvex representation.

\begin{definition}\label{hyperconvex} \textup{(}\cite{PSW}\textup{)} Let $\rho:\Gamma \rightarrow \mathsf{GL}_{d}(\mathbb{K})$ be a $\{p,q,r\}$-Anosov representation, where $p+q\leq r$. The representation $\rho$ is called $(p,q,r)$-hyperconvex if for every triple $x,y,w \in \partial_{\infty}\Gamma$ of distinct points we have: $$\big(\xi_{\rho}^p(x)\oplus \xi_{\rho}^q(y)\big)\cap \xi_{\rho}^{d-r}(w)=(0).$$\end{definition}

\begin{proof}[Proof of Theorem \ref{hyp}] Suppose that $\rho:\Gamma \rightarrow \mathsf{GL}_d(\mathbb{K})$ is $(p,p,2p)$-hyperconvex and fix $w_0\in \partial_{\infty}\Gamma$. For $x\neq w_0$, $\tau_p(\xi_{\rho}^p(x))$ is transverse to the hyperplane $\tau_{d-p}(\xi_{\rho}^{d-p}(w_0))$, thus, by Observation \ref{lift}, $\tau_p\circ \xi_{\rho}^p:\partial_{\infty}\Gamma \smallsetminus \{w_0\}\rightarrow \mathbb{P}(\wedge^p \mathbb{K}^d)$ lifts to a well defined continuous map $f_p:\partial_{\infty}\Gamma \smallsetminus\{w_0\} \rightarrow \mathsf{S}^1(\wedge^p\mathbb{K}^d)$. After fixing a lift $V_0\in \wedge^{d-2p}\mathbb{K}^{d}$ of $\tau_{d-2p}(\xi_{\rho}^{d-2p}(w_0))$, consider $\mathcal{H}:(\partial_{\infty}\Gamma \smallsetminus \{w_0\})\times (\partial_{\infty}\Gamma \smallsetminus \{w_0\})\rightarrow \mathbb{K}$ defined as follows $$\mathcal{H}(x,y):=f_p(x)\wedge f_p(y)\wedge V_0.$$ The map $\mathcal{H}$ is continuous, $\mathcal{H}(x,y)\neq 0$ for every $x\neq y$ since $\xi_{\rho}^p(x)\oplus \xi_{\rho}^p(y)\oplus \xi_{\rho}^{d-2p}(w_0)=\mathbb{K}^d$ and \begin{align*}\mathcal{H}(y,x)=f_p(y)\wedge f_p(x)\wedge V_0=(-1)^{p^2}f_p(x)\wedge f_p(y)\wedge V_0=-\mathcal{H}(x,y) \end{align*} because $p$ is odd.
\medskip

\noindent \textup{(i)} $\mathbb{K}=\mathbb{R}$. Suppose that $\Gamma$ is not virtually a surface group or a free group. The main results from \cite{Gabai} and \cite{BK} imply that there exists an embedding $j_1:S^1\rightarrow \partial_{\infty}\Gamma \smallsetminus \{w_0\}$. Then the map $h_1:S^1\rightarrow \mathbb{R}$, $$h_1(z)=\mathcal{H}\big(j_1(z),j_1(-z)\big) \ \ z \in S^1,$$ is odd and nowhere vanishing, a contradiction. It follows that $\Gamma$ is virtually a free group or virtually a surface group.
\medskip

\noindent \textup{(ii)} $\mathbb{K}=\mathbb{C}$. Suppose that $\partial_{\infty}\Gamma$ properly contains a $2$-sphere. Fix an embedding $\iota_2:S^2\rightarrow \partial_{\infty}\Gamma \smallsetminus \{w_0\}$ and we see that the map $h_2:S^2\rightarrow \mathbb{C}$, $$h_2(z)=\mathcal{H}\big(j_2(z),j_2(-z)\big) \ \ z\in S^2,$$ is odd and nowhere vanishing, contradicting the Borsuk--Ulam theorem. This shows that $\partial_{\infty}\Gamma$ cannot contain a $2$-sphere.\end{proof}

\end{document}